\documentclass[11pt]{amsart}
\usepackage[all]{xy}
\usepackage{a4wide,fullpage}
\usepackage{amssymb}
\usepackage{amsthm}
\usepackage{amsmath}
\usepackage{amscd,enumitem}
\usepackage{mathrsfs}
\usepackage{verbatim}
\usepackage{float}
\usepackage{color}
\usepackage{bbm}
\usepackage[all]{xy}
\usepackage{hyperref}

\theoremstyle{plain}
\newtheorem{theorem}{Theorem}

\newtheorem{proposition}[theorem]{Proposition}
\newtheorem*{theorem*}{Theorem}

\theoremstyle{definition}

\theoremstyle{remark}
\newtheorem*{remark}{Remark}

\newcommand{\Z}{\mathbb{Z}}
\newcommand{\N}{\mathbb{N}}

\newcommand{\im}{\operatorname{Im}}

\numberwithin{equation}{section}
\numberwithin{theorem}{section}

\allowdisplaybreaks

\begin{document}
\title{Interesting identities involving weighted representations of integers as sums of arbitrarily many squares}
\author{Min-Joo Jang}
\address{Department of Mathematics, University of Hong Kong, Pokfulam, Hong Kong}
\email{min-joo.jang@hku.hk}
\author{Ben Kane}
\address{Department of Mathematics, University of Hong Kong, Pokfulam, Hong Kong}
\email{bkane@hku.hk}
\author{Winfried Kohnen}
\address{Mathematisches Institut, Universit\"at Heidelberg, INF 288, 69210, Heidelberg, Germany}
\email{winfried@mathi.uni-heidelberg.de}
\author{Siu Hang Man}
\address{Mathematisches Institut, Georg-August-Universit\"at G\"ottingen, Wilhelmsplatz 1, 37073 G\"ottingen, Germany}
\email{siu-hang.man@mathematik.uni-goettingen.de}
\thanks{The research of the second author was supported by grants from the Research Grants Council of the Hong Kong SAR, China (project numbers HKU 17302515, 17316416, 17301317, and 17303618).}
\begin{abstract}
We consider the number of ways to write an integer as a sum of squares, a problem with a long history going back at least to Fermat. The previous studies in this area generally fix the number of squares which may occur and then either use algebraic techniques or connect these to coefficients of certain complex analytic functions with many symmetries known as modular forms, from which one may use techniques in complex and real analysis to study these numbers. In this paper, we consider sums with arbitrarily many squares, but give a certain natural weighting to each representation. Although there are a very large number of such representations of each integer, we see that the weighting induces massive cancellation and we furthermore prove that these weighted sums are again coefficients of modular forms, giving precise formulas for them in terms of sums of divisors of the integer being represented.
\end{abstract}
\date{\today}
\maketitle

\section{Introduction and statement of results}

There are many papers involving representations of integers as sums of a fixed number of squares and their relations with modular forms. In this paper, we consider representations of integers as the sum of arbitrarily many squares and relate this to quasimodular forms by weighting the representations in a certain way. To state the result, for a positive integer $n$, let $r^{\operatorname{e}}(n)$ (resp. $r^{\operatorname{o}}(n)$) denote the number of ways to write $n$ as the sum of an even (resp. odd) number of non-zero squares (taking $r^{\operatorname{e}}(0):=1$ and $r^{\operatorname{o}}(0):=0$ by convention because we include empty representations of $n=0$ in our count throughout the paper). We next see that the difference of these, namely $r^{\operatorname{e}-\operatorname{o}}(n):=r^{\operatorname{e}}(n)-r^{\operatorname{o}}(n)$, has an interesting relationship with certain sums of divisor functions which naturally arise as Fourier coefficients of weight two (quasi-)modular forms.
\begin{theorem}\label{thm:square}
For every $n\in \N$ we have 
\[
\sum_{m\in \Z} m^2 r^{\operatorname{e}-\operatorname{o}}\left(n-m^2\right)=\sum_{d\mid n} dc_d,
\] 
where 
\[
c_n=
\begin{cases}
2&\text{if $n\equiv 1 \pmod{2}$},\\
-3&\text{if $n\equiv 2 \pmod{4}$},\\
-1&\text{if $n\equiv 0 \pmod{4}$}.\\
\end{cases}
\]
Equivalently, writing $n=2^t n'$ with $2\nmid n'$, we have
\[
\sum_{m\in\Z} m^2 r^{\operatorname{e}-\operatorname{o}}(n-m^2)=\left(2-6\delta_{t>0}\right)\sum_{d\mid n'} d - 4\sum_{d\mid \frac{n}{4}} d.
\]
\end{theorem}
\begin{remark}
Enumerating the representatations of $n$ as a sum of arbitrarily many squares in the shape (for some $\ell \geq 0$ and $m_j\in \Z$)
\[
n=m^2 + \sum_{j=1}^{\ell} m_j^2,
\]
one can interpret the sum occurring on the left-hand side of Theorem \ref{thm:square} as a sum over all such representations weighted by $(-1)^{\ell}m^2$, where $m^2$ is a ``distinguished'' square occurring in the representation (note that every distinct $m_j^2$ plays the role of the ``distinguished'' square as we run through all $m$). 
\end{remark}
The phenomenon observed in Theorem \ref{thm:square} is not an isolated one. A similar result holds for representations of $n$ as a sum of arbitrarily many triangular numbers. Let $t^{\operatorname{e},+}(n)$ (resp. $t^{\operatorname{o},+}(n)$ denote the number of representations of $n$ as a sum 
\[
n=\sum_{j=1}^{\ell} T_{x_j}
\]
with $\ell$ even (resp. $\ell$ odd) where $T_{x}:=\frac{x(x+1)}{2}$, with $x\in\N$, denotes the $x$-th triangular number; throughout this paper we assume for triangular numbers that $x\in\N$ in order to remove ambiguity in the representations coming from the fact that $T_{-1}=T_0=0$.  Analogously to the case of squares, we set $t^{\operatorname{e}-\operatorname{o},+}(n):=t^{\operatorname{e},+}(n)-t^{\operatorname{o},+}(n)$ to be the excess of the number of such representations with an even number of triangular numbers. Here and throughout, we use `$+$' in a superscript to indicate that we only take elements in $\N$ instead of $\Z\setminus\{0\}$. 
\begin{theorem}\label{thm:triangular}
For every $n\in \N$, we have 
\[
\sum_{m\ge0} \frac{(2m+1)^2}{8}t^{\operatorname{e}-\operatorname{o},+}\left(n-T_m\right)=\sum_{d\mid n} (-1)^{d+1} d.
\]
\end{theorem}
Moreover, the phenomenon is not restricted to squares or triangular numbers, but occurs for arbitrary representations of integers as sums of generalized polygonal numbers. For $k\in \N$ and $m\in\Z$, let 
\[
p_k(m):=\frac{(k-2)m^2-(k-4)m}{2}
\]
be the $m$-th \begin{it}generalized $k$-gonal number\end{it}. Following the definitions in the previous two cases, we define $P_k^{\operatorname{e}-\operatorname{o}}(n):= P_{k}^{\operatorname{e}}(n)- P_{k}^{\operatorname{o}}(n)$ with $P_{k}^{\operatorname{e}}(n)$ ($P_k^{\operatorname{o}}(n)$ resp.) being the number of ways to write $n$ as a sum of an even (odd resp.) number of non-zero generalized $k$-gonal numbers. 
\begin{theorem}\label{thm:kgonal} 
For $k>4$ we have 
\[
\sum_{m\in\Z}
\left(p_k(m)+\frac{(k-4)^2}{8(k-2)}\right)P_k^{\operatorname{e}-\operatorname{o}}\left(n-p_k(m)\right) =\sum_{d\mid n} de'_d,
\]
where
\[
e'_n=\begin{cases}
1&\text{if $n\equiv \pm(k-3) \pmod{k-2}$},\\
-1&\text{if $n\equiv 0, (k-2), \pm2(k-3)   \pmod{2(k-2)}$},\\
0&\text{otherwise.}
\end{cases}
\]
\end{theorem}
\begin{remark}
Although the formulas in Theorem \ref{thm:kgonal} do not hold for $k=4$ and $k=3$, it is worth noting that one can interpret the formulas 
in Theorems \ref{thm:square}, \ref{thm:triangular}, and \ref{thm:kgonal} uniformly by counting the instances of the congruences occurring in Theorem \ref{thm:kgonal} with multiplicity.

\end{remark}

By completing the square, one may interpret the statements in Theorems \ref{thm:square}, \ref{thm:triangular}, and \ref{thm:kgonal} as statements about representations of integers as sums of squares with certain congruence conditions. Indeed, these formulas hold for representations of integers as sums of squares with arbitrary congruence conditions. Define $r_{h,N}^{\operatorname{e}}(n)$ (resp. $r_{h,N}^{\operatorname{o}}(n)$)  for $N\in\N$ and $h\in\Z$ satisfying $|h|<\frac{N}{2}$ by 
\begin{equation}\label{eqn:setdef}
\#\left\{(s_1, s_2,\dots, s_\ell)\in\left(\Z\setminus\{0\}\right)^\ell,\ \ell \text{ even (resp. odd) } \Bigg| \ \sum_{j=1}^\ell \left(s_j^2N^2+2s_j hN\right)=n \right\}.
\end{equation}
We then denote $r_{h,N}^{\operatorname{e}-\operatorname{o}}(n):=r_{h,N}^{\operatorname{e}}(n)-r_{h,N}^{\operatorname{o}}(n)$ as before. 
\begin{theorem}\label{thm:general}
Suppose $N\neq \pm 2h$ and $h\neq0$. We have 
\[
\sum_{m\equiv h\!\!\!\!\!\pmod{N}} m^2 r_{h,N}^{\operatorname{e}-\operatorname{o}}\left(n-m^2+h^2\right)=\sum_{d\mid n} de_d,
\]
where 
\[
e_n=\begin{cases}
1&\text{if $n\equiv \pm(N^2+2hN) \pmod{2N^2}$},\\
-1&\text{if $n\equiv 0, 2N^2, \pm(2N^2+4hN) \pmod{4N^2}$},\\
0&\text{otherwise.}
\end{cases}
\]
\end{theorem}
  We write $\sum_{m\equiv h\!\!\pmod{N}}$, $\sum'_{m\equiv h\!\!\pmod{N}}$, and $\sum_{m\equiv h\!\!\pmod{N}}^+$ to mean the sums over the underlying sets $m\in\Z$, $m\in\Z\setminus\{h\}$, and $m\in\N$ respectively.
\begin{remark}
In the case $N=\pm 2h$, a modified version of Theorem \ref{thm:general} (see the proof of Theorem \ref{thm:triangular}) holds where $r_{h,N}^{\operatorname{e}-\operatorname{o}}(n)$ is replaced with $r_{h,N}^{\operatorname{e}-\operatorname{o},+}(n)$, where as in the definition of $t^{\operatorname{e}-\operatorname{o},+}(n)$, the $+$ indicates that we only take $s_j\in\N$ in the definition \eqref{eqn:setdef}. For $h=0$, one may consider Theorem \ref{thm:general} to hold if one interprets $n\equiv 2N^2\pmod{4N^2}$ to occur with multiplicity $3$ and $n\equiv N^2\pmod{2N^2}$ to occur with multiplicity $2$ in the congruences in the definition of $e_n$.

\end{remark}
The proofs of the above theorems make use of product formulas for the unary theta functions (writing $q:=e^{2\pi i\tau}$)
\[
\vartheta_{h,N}(\tau) := \sum\limits_{n\equiv h\!\!\!\!\!\pmod{N}} q^{n^2}. 
\]
Indeed, the $c_n$ appearing in Theorem \ref{thm:square} are closely related to the well-known identity
\[
\theta(\tau)=\prod_{n\geq 1} \frac{\left(1-q^{2n}\right)^5}{\left(1-q^{4n}\right)^2\left(1-q^n\right)^2}=\prod_{n\geq 1}(1-q^n)^{-c_n}
\]
for the standard theta function 
\[
\theta(\tau):=\sum_{n\in\Z} q^{n^2}.
\]
Product formulas such as these naturally appeared in work of Bruinier, Ono, and the third author \cite{BKO} while investigating divisors of modular forms. A detailed investigation of these product formulas was then carried by Mason and the third author \cite{KM}. Mason and the third author obtained a recursive identity involving multinomial coefficients and the Fourier coefficients of the original modular form. Namely,
\[
c_n=2a(n)-\sum_{\substack{\nu_1+2\nu_2+\cdots +(n-1)\nu_{n-1}=n\\ \nu_1,\nu_2,\dots,\nu_{n-1}\ge0}}(-1)^{\nu_1+\nu_2+\cdots+\nu_{n-1}}\binom{c_1}{\nu_1}\binom{c_2}{\nu_2}\cdots\binom{c_{n-1}}{\nu_{n-1}},
\]
where $a(n)$ is the $n$-th Fourier coefficient. It would be interesting to see if one can use the representation from \cite{KM} in order to directly prove the identities, such as the one in Theorem \ref{thm:square}. Conversely, it may be interesting to further investigate whether the two different representations yield new identities between different combinatorial objects.

Theorem \ref{thm:square} implies that the sum giving $c_n$ in \cite{KM} exhibits massive cancellation, as one can see that there are exponentially many terms (as a function of $n$) in the sum, but the value is a constant depending only on $n$ modulo $4$. This cancellation is very delicate, and relies on the fact that the theta functions $\vartheta_{h,N}$ do not vanish in the upper half-plane. Even slightly modifying the theta function will yield a function for which the series defining $c_n$ does not exhibit this behavior. One such family of functions occurs by twisting the coefficients by a character $\chi$, defining 
\[
\theta_{\chi}(\tau):=\sum_{n\in\Z} \chi(n) q^{n^2}.
\]
However, Lemke Oliver \cite{LO} proved that there are only finitely many characters $\chi$ for which $\theta_{\chi}$ may be written as an $\eta$-quotient. 

Since this cancellation is governed by the fact that $\vartheta_{h,N}$ vanishes only at the cusps (i.e., it does not have a root in the upper half-plane), from this perspective it is natural to wonder whether one can directly prove that $\vartheta_{h,N}$ only vanishes at the cusps. An amusing calculation involving the full modularity properties of $\vartheta_{h,N}$ as given in \cite{Shimura} together with the valence formula and a careful evaluation of the resulting Gauss sums leads to such a proof (at least in the case when $N$ is the power of an odd prime).

\section{Proof of the theorems}

Theorems \ref{thm:square}, \ref{thm:triangular}, \ref{thm:kgonal}, and \ref{thm:general} are all essentially corollaries of one proposition. To state the proposition, for a given $h,N$ with $|h|\leq \frac{N}{2}$, we define $a_n=a_n(h,N)$ as follows.
  If $h\neq0$ and $N\neq \pm 2h$, then we have
\begin{equation}\label{eqn:e_n}
a_n:=e_n=\begin{cases}
1&\text{if $n\equiv \pm(N^2+2hN) \pmod{2N^2}$},\\
-1&\text{if $n\equiv 0, 2N^2, \pm(2N^2+4hN) \pmod{4N^2}$},\\
0&\text{otherwise.}
\end{cases}
\end{equation}
Similarly, for $h=0$ 
\begin{equation}\label{eqn:c_n}
a_n:=c_n=
\begin{cases}
2&\text{if $n\equiv N^2 \pmod{2N^2}$},\\
-3&\text{if $n\equiv 2N^2 \pmod{4N^2}$},\\
-1&\text{if $n\equiv 0 \pmod{4N^2}$},\\
0&\text{otherwise,}
\end{cases}
\end{equation}
while for $N=\pm2h$
\begin{equation}\label{eqn:f_n}
a_n:=f_n=
\begin{cases}
1&\text{if $n\equiv 8h^2 \pmod{16h^2}$},\\
-1&\text{if $n\equiv 0 \pmod{16h^2}$},\\
0&\text{otherwise.}
\end{cases}
\end{equation}
\begin{proposition}\label{prop:gen}
Let $h,N$ be given with $|h|\leq \frac{N}{2}$.
\noindent

\noindent
\begin{enumerate}[leftmargin=*,align=left,label={\rm(\arabic*)}]
\item 
If $|h|<\frac{N}{2}$, then we have 
\begin{equation}\label{eqn:compareCoeff}
\sum_{m\equiv h\!\!\!\!\!\pmod{N}}m^2 {r}_{h,N}^{\operatorname{e}-\operatorname{o}}\left(n-m^2+h^2\right)
=\sum_{d\mid n} d a_d.
\end{equation}
\item 
If $N=\pm 2h$, then we have
\begin{equation}\label{eqn:compareCoeff3}
\sideset{}{^+}{\sum}_{m\equiv h\!\!\!\!\!\pmod{N}}m^2 {r}_{h,N}^{\operatorname{e}-\operatorname{o},+}\left(n-m^2+h^2\right)=\sum_{d\mid n}da_d.
\end{equation}
\end{enumerate}
\end{proposition}
\begin{proof}
Applying the Jacobi triple product identity gives that
\begin{align}
\vartheta_{h,N}(\tau)&=q^{h^2}\prod_{n\ge1}\left(1+q^{(2n-1)N^2+2hN}\right)\left(1+q^{(2n-1)N^2-2hN}\right)\left(1-q^{2nN^2}\right)\notag\\
&=\left(1+\delta_{N=\pm2h}\right)q^{h^2}\prod_{n\ge1}(1-q^n)^{-a_n}.\label{eqn:prod}
\end{align}
Taking the logarithmic derivative of \eqref{eqn:prod}, we obtain 
\begin{equation}\label{eqn:Thetalogdiff}
\frac{1}{2\pi i} \frac{\vartheta'_{h,N}(\tau)}{\vartheta_{h,N}(\tau)} = h^2 +\sum_{n\geq 1} \frac{na_nq^n}{1-q^n} = h^2+\sum_{n\geq 1}\sum_{m\geq 1}  na_nq^{mn}= h^2 +\sum_{n\geq 1} \sum_{d\mid n} d a_d q^n. 
\end{equation}

\noindent
(1) Under the assumption that $|h|<\frac{N}{2}$, we write
\[
\vartheta_{h,N}(\tau)=q^{h^2}\left(1-\left(1-q^{-h^2}\vartheta_{h,N}(\tau)\right)\right)=q^{h^2}\left(1-\left(-\sideset{}{'}\sum_{n\equiv h\!\!\!\!\!\pmod{N}} q^{n^2-h^2}\right)\right)
\]
 and note that
\[
\frac{1}{2\pi i} \vartheta'_{h,N}(\tau) = \sum_{n\equiv h\!\!\!\!\!\pmod{N}} n^2q^{n^2}.
\]
Since $|h|<\frac{N}{2}$, $h$ is the unique element among $n\equiv h\pmod{N}$ with $n^2$ minimal. Hence for $v:=\im(\tau)$ sufficiently large we have $\left|1-q^{-h^2}\vartheta_{h,N}(\tau)\right|<1$, and thus one may expand the geometric series to obtain 
\begin{align}
\frac{1}{2\pi i}\frac{\vartheta'_{h,N}(\tau)}{\vartheta_{h,N}(\tau)} &= \frac{q^{-h^2}}{2\pi i}\frac{\vartheta'_{h,N}(\tau)}{1-(1-q^{-h^2}\vartheta_{h,N}(\tau))} \notag\\
&=q^{-h^2}\sum_{n\equiv h\!\!\!\!\!\pmod{N}}  n^2 q^{n^2} \sum_{\ell\geq 0} \left(-\sideset{}{'}\sum_{n\equiv h\!\!\!\!\!\pmod{N}} q^{n^2-h^2}\right)^{\ell}\notag\\
&=q^{-h^2}\sum_{n\equiv h\!\!\!\!\!\pmod{N}}n^2q^{n^2} \sum_{m\ge0}\sum_{\ell\geq 0} (-1)^{\ell}\widetilde{r}_{h,N,\ell}\left(m+\ell h^2\right)q^m,\label{eqn:geom1}
\end{align}
where 
\begin{equation}\label{eqn:tilderdef}
\widetilde{r}_{h,N,\ell}(m+\ell h^2):=\#\left\{(k_1, k_2,\dots, k_\ell)\in\left(N\Z\setminus\{0\}+h\right)^\ell \Bigg| \ \sum_{j=1}^\ell k_j^2=m+\ell h^2 \right\}.
\end{equation}
Writing $k_j=s_jN+h$ with $s_j\in\Z\setminus\{0\}$ yields that 
\[
\sum_{j=1}^\ell k_j^2=m+\ell h^2\iff \sum_{j=1}^\ell \left(s_j^2N^2+2s_j hN\right)=m.
\]
Therefore, with notation as above, we may write \eqref{eqn:geom1} as 
\begin{equation}\label{eqn:geom2}
\frac{1}{2\pi i}\frac{\vartheta'_{h,N}(\tau)}{\vartheta_{h,N}(\tau)} =\sum_{n\ge0}\sum_{m\equiv h\!\!\!\!\!\pmod{N}}m^2 {r}_{h,N}^{\operatorname{e}-\operatorname{o}}\left(n-m^2+h^2\right)q^n.
\end{equation}
Comparing Fourier coefficients of \eqref{eqn:Thetalogdiff} with \eqref{eqn:geom2}, we obtain \eqref{eqn:compareCoeff}.

\noindent
(2) When $N=\pm 2h$, we note that $h$ and $-h$ are both minimal elements in terms of absolute value in $n\equiv h\pmod{N}$. It is hence natural to quotient out by the symmetry and we modify the definition of $\widetilde{r}_{h,N,\ell}$ in this case by defining $\widetilde{r}_{h,N,\ell}^+(m+\ell h^2)$ to be the number of elements in \eqref{eqn:tilderdef} with $k_j\in N\N+h$ for every $1\leq j\leq \ell$. To factor out by the symmetry, we rewrite $\vartheta_{h,2|h|}(\tau)$ as
\[
\vartheta_{h,2|h|}(\tau)=\sum_{n\equiv h\!\!\!\!\!\pmod{2|h|}} q^{n^2}=2\sideset{}{^+}\sum_{n\equiv h\!\!\!\!\!\pmod{2|h|}} q^{n^2}=2\sideset{}{^+}\sum_{n\text{ odd}} q^{h^2n^2},
\]
and, expanding the geometric series after accounting for the two minimal elements, 
\begin{align}
\frac{1}{2\pi i}\frac{\vartheta'_{h,2|h|}(\tau)}{\vartheta_{h,2|h|}(\tau)}
&= \frac{q^{-h^2}}{4\pi i}\frac{\vartheta'_{h,N}(\tau)}{1-\left(1-\frac{q^{-h^2}\vartheta_{h,N}(\tau)}{2}\right)}\notag\\
&=q^{-h^2}\sideset{}{^+}\sum_{n\text{ odd}}n^2q^{h^2n^2} \sum_{m\ge0}\sum_{\ell\geq 0} (-1)^{\ell}\widetilde{r}_{h,2|h|,\ell}^+\left(m+\ell h^2\right)q^m\notag\\
&=q^{-h^2}\sideset{}{^+}\sum_{n\text{ odd}}n^2q^{h^2n^2} \sum_{m\ge0}r^{\operatorname{e}-\operatorname{o},+}_{h,2|h|}\left(m\right)q^m\notag\\
&=\sum_{n\ge0}\sum_{m\ge0}h^2(2m+1)^2r^{\operatorname{e}-\operatorname{o},+}_{h,2|h|}\left(n-4h^2(m^2+m)\right)q^n. \label{eqn:geom3}
\end{align}
Comparing Fourier coefficients of \eqref{eqn:geom3} with \eqref{eqn:Thetalogdiff} together with \eqref{eqn:f_n} yields \eqref{eqn:compareCoeff3}.

\end{proof}

Theorem \ref{thm:general} now follows immediately from Proposition \ref{prop:gen} (1).
\begin{proof}[Proof of Theorem \ref{thm:general}]
When $h\neq 0$ and $N\neq \pm 2h$,  plugging \eqref{eqn:e_n} into \eqref{eqn:compareCoeff} yields Theorem \ref{thm:general}.
\end{proof}

We next move on to the special case of sums of $k$-gonal numbers in Theorems \ref{thm:square}, \ref{thm:triangular}, and \ref{thm:kgonal}, which are essentially nice ways to rewrite Proposition \ref{prop:gen}. We begin with the general case. 
\begin{proof}[Proof of Theorem \ref{thm:kgonal}]
Setting $N=2(k-2),\ h=4-k$ and $n\mapsto 8(k-2)n$ in \eqref{eqn:compareCoeff} and 
noting that
\[
r^{\operatorname{e}-\operatorname{o}}_{4-k,2(k-2)}(8(k-2)n)=P_k^{\operatorname{e}-\operatorname{o}}(n),
\] 
we have
\begin{align}
\sum_{m\equiv 4-k\!\!\!\!\pmod{2(k-2)}}& m^2 r_{4-k,2(k-2)}^{\operatorname{e}-\operatorname{o}}\left(8(k-2)n-m^2+(k-4)^2\right)
\notag\\
&=\sum_{m\in\Z}8(k-2)\left(p_k(m)+
\frac{(k-4)^2}{8(k-2)}
\right)P_k^{\operatorname{e}-\operatorname{o}}\left(n-p_k(m)\right)=\sum_{d\mid 8(k-2)n} de_d,\label{eqn:compareCoeff2}
\end{align}
where
\[
e_n=\begin{cases}
1&\text{if $n\equiv \pm8(k-2)(k-3) \pmod{8(k-2)^2}$},\\
-1&\text{if $n\equiv 0, 8(k-2)^2, \pm16(k-2)(k-3) \pmod{16(k-2)^2}$},\\
0&\text{otherwise.}
\end{cases}
\]
Notice that $e_n=0$ unless $8(k-2)\mid n$. Making the change of variables $d\mapsto 8(k-2)d$ in \eqref{eqn:compareCoeff2} and divididng both sides by $8(k-2)$, we complete the proof of Theorem \ref{thm:kgonal}.
\end{proof}

The $h=0$ case of Proposition \ref{prop:gen} (1) corresponds to the case of squares.

\begin{proof}[Proof of Theorem \ref{thm:square}]
For $h=0$, we rewrite \eqref{eqn:compareCoeff} using \eqref{eqn:c_n} as 
\[
\sum_{m\in N\Z}m^2 {r}_{0,N}^{\operatorname{e}-\operatorname{o}}\left(n-m^2\right)=\sum_{d\mid n} d c_d.
\]
Replacing $n$ by $N^2n$, and noting that 
\[
r_{0,N}^{\operatorname{e}-\operatorname{o}}(N^2n)=r^{\operatorname{e}-\operatorname{o}}(n),
\]
 we obtain
\[
\sum_{m\in \Z}N^2m^2 {r}_{0,N}^{\operatorname{e}-\operatorname{o}}\left(N^2n-N^2m^2\right)=\sum_{m\in \Z}N^2m^2 {r}^{\operatorname{e}-\operatorname{o}}\left(n-m^2\right)=\sum_{d\mid N^2n} d c_d.
\]
Moreover, since $c_d=0$ if $N^2\nmid d$, the sum on the right-hand side equals
\[
\sum_{d\mid N^2n} d c_d=\sum_{d\mid n} N^2 d c'_d,
\] 
where
\[
c'_n=
\begin{cases}
2&\text{if $n$ odd},\\
-3&\text{if $n\equiv 2 \pmod{4}$},\\
-1&\text{if $4\mid n$}.
\end{cases}
\]
Dividing both sides by $N^2$ gives Theorem \ref{thm:square}.
\end{proof}

Finally, we obtain the claim for sums of triangular numbers by using Proposition \ref{prop:gen} (2).
\begin{proof}[Proof of Theorem \ref{thm:triangular}]
Note that
\[
r^{\operatorname{e}-\operatorname{o},+}_{h,2|h|}(8h^2n-4h^2(m^2+m))=
r^{\operatorname{e}-\operatorname{o},+}_{1,2}(8n-4(m^2+m))=t^{\operatorname{e}-\operatorname{o},+}\left(n-T_m\right).
\]
Using this and replacing $n$ by $8h^2n$ in \eqref{eqn:compareCoeff3} yields that
\[
\sum_{m\ge0}h^2(2m+1)^2 t^{\operatorname{e}-\operatorname{o},+}\left(n-T_m\right)=\sum_{d\mid 8h^2n}df_d=\sum_{d\mid n} 8h^2 df'_d,
\]
where
\[
f'_n=\begin{cases}
1&\text{if $n$ odd},\\
-1&\text{if $n$ even}.\\
\end{cases}
\]
The last equality is valid because $f_d=0$ unless $8h^2\mid d$. Dividing both sides by $8h^2$, we have 
\[
\sum_{m\ge0}\frac{(2m+1)^2}{8}t^{\operatorname{e}-\operatorname{o},+}\left(n-T_m\right)=\sum_{d\mid n} df'_d= \sum_{d\mid n}(-1)^{d+1} d,
\]
as claimed in Theorem \ref{thm:triangular}.
\end{proof}

\begin{section}*{Acknowledgements}
The authors thank the anonymous referee for their careful reading and helpful comments on the earlier version of this paper.
\end{section}

\end{document}